\documentclass[11pt,reqno]{amsart}
\usepackage{amssymb}
\usepackage{amsfonts}
\usepackage{amsthm}
\usepackage{amscd}
\numberwithin{equation}{section}
\theoremstyle{plain}
\newtheorem{theorem}{Theorem}[section]
\newtheorem{proposition}[theorem]{Proposition}
\newtheorem{corollary}[theorem]{Corollary}
\newtheorem{lemma}[theorem]{Lemma}
\theoremstyle{definition}

\newtheorem{remark}[theorem]{Remark}

\newcommand{\refE}[1]{(\ref{E:#1})}
\newcommand{\refS}[1]{Section~\ref{S:#1}}

\newcommand{\refT}[1]{Theorem~\ref{T:#1}}

\newcommand{\refP}[1]{Proposition~\ref{P:#1}}

\newcommand{\refR}[1]{Remark~\ref{R:#1}}
\newcommand{\refC}[1]{Corollary~\ref{C:#1}}
\newcommand{\refL}[1]{Lemma~\ref{L:#1}}

\newcommand{\C}{\ensuremath{\mathbb{C}}}
\newcommand{\N}{\ensuremath{\mathbb{N}}}

\newcommand{\Pro}{\ensuremath{\mathbb{P}}}
\newcommand{\Z}{\ensuremath{\mathbb{Z}}}
\newcommand{\K}{\ensuremath{\mathbb{K}}}

\newcommand{\W}{\mathcal{W}}
\newcommand{\V}{\mathcal{V}}

\newcommand{\Ho}{\mathrm{H}}
\newcommand{\chara}{\mathrm{char}}

\newcommand{\ldot}{\,.\,} 

\begin{document}
\title[Vanishing of the second cohomology of the Virasoro algebra] 
{An elementary proof of the vanishing of the 
second cohomology of the Witt and Virasoro algebra with values in the
adjoint module}
\author[Martin Schlichenmaier]{Martin Schlichenmaier}
\thanks{This work was finalized during a visit to the  Institut
  Mittag-Leffler
(Djursholm, Sweden) in the frame of the program
``Complex Analysis and
  Integrable
Systems'', Fall 2011. Partial  support by
the ESF networking programme HCAA, and the
Internal Research Project  GEOMQ08,  University of Luxembourg,
is acknowledged.}
\address{%
University of Luxembourg\\
Mathematics Research Unit, FSTC\\
Campus Kirchberg\\ 6, rue Coudenhove-Kalergi,
L-1359 Luxembourg-Kirchberg\\ Luxembourg
}
\email{martin.schlichenmaier@uni.lu}
\begin{abstract}
By elementary and direct calculations the 
vanishing of the (algebraic) second
Lie algebra cohomology of the Witt and the Virasoro algebra with 
values in the adjoint module is shown.
This yields infinitesimal and formal rigidity or these 
algebras.
The first (and up to now only) proof of this 
important result 
 was given 1989 by Fialowski in
an unpublished note.
It is based on cumbersome calculations.
 Compared to the original proof
the presented one is quite elegant and considerably simpler.
\end{abstract}
\date{28.11.2011, rev. 27.1.2012, reference update after publication 
8.5.2012}
\maketitle

\vskip 1.0cm
\section{Introduction}\label{S:intro}
The simplest  nontrivial 
infinite dimensional Lie algebras are the  Witt algebra 
and its central extension the Virasoro algebra.
The Witt algebra is related to the Lie algebra  of 
the group of diffeomorphisms of the unit circle.
The central extension comes into play as one 
is typically forced to consider projective actions if one
wants to quantize a classical system or wants to regularize a
field theory.
There is a huge amount of literature about the
application of these algebras. Here it is not the place to
give even a modest  overview about these applications.

The { Witt algebra} $\W$ is the 
graded Lie algebra  generated as vector space by the
elements $\{e_n\mid n\in \mathbb{Z}\}$ with 
Lie structure  
\begin{equation}
[e_n,e_m]=(m-n)e_{n+m},\quad n,m\in\Z.
\end{equation}
The {Virasoro algebra} $\V$ is  the universal one-dimensional 
central extension. Detailed definitions and descriptions are
given in \refS{algebra}.
The goal of this article is to 
give an elementary proof that the
second Lie algebra cohomology of both algebras  with 
values in the adjoint modules 
will vanish.
The result will be stated in 
\refT{main} below.
These cohomology spaces are related to deformations of
the algebras. 
In particular, if they  vanish 
the algebras will be infinitesimal and formally rigid,
see \refC{rigid}.
Note that this does not mean that the algebras are
analytically or geometrically  rigid. 
Indeed, together with Alice Fialowski we showed in
\cite{FiaSchlV} that there 
exist natural, geometrically defined, 
nontrivial families of Lie algebras given by
Krichever-Novikov type algebras 
associated to  elliptic curves.
These families appeared 
in \cite{SchlDeg}.
In these families
the special fiber is the Witt (resp. 
Virasoro) 
algebra but all other fibers are non-isomorphic to 
it. Similar results for the current, resp. affine Lie algebras
can be found in \cite{FiaSchlaff},  
\cite{FiaSchlaffo}.

The result on the vanishing of the second Lie algebra cohomology
of the Witt algebra should clearly be attributed to 
Fialowski. There exists an unpublished manuscript
\cite{Fianote} 
by her, dating from 1989, where she does  explicit calculations
\footnote{Note added after publication of the current
article: The content of the mentioned 
note of Fialowski is in the meantime accessible on the archive 
and  will appear in Jour. Math. Phys. See \cite{Fianote} for 
details.}.
These calculations were quite cumbersome not really
 appealing for journal publication.
In 1990 Fialowski gave the statement of the rigidity of the Witt 
and Virasoro algebra in \cite{Fiajmp} without proof.
Later, in the above-mentioned joint paper of the author with her
\cite{FiaSchlV} a sketch of its  proof was presented.
It is based on the calculations
of the cohomology of the Lie algebra $Vect(S^1)$ of vector fields on
$S^1$ with values in the adjoint module.
Hence, this proof is not purely algebraic.
Based on results of Tsujishita \cite{Tsu}, Reshetnikov \cite{Resc},  
and Goncharova \cite{Gon} we showed   that  
\begin{equation}
\Ho^*(Vect(S^1),Vect(S^1))=\{0\}.
\end{equation}
See also the book of Guieu and Roger \cite{GR}.
We argued that by density arguments 
the vanishing of the cohomology of the Witt algebra will follow.
This is indeed true if one considers 
continuous cohomology.
But here we are dealing with 
arbitrary cohomology. 
In a recent attempt to fill  the details how to
extend the 
arguments to this setting, 
we run into troubles.  
We would have been forced to make considerable 
cohomological changes to obtain finally certain boundedness 
properties  to guarantee convergence with respect to the
topology coming from the Lie algebra of vector fields.
Hence, there was an incentive to return to a direct algebraic 
and elementary proof. Indeed, I found a very elementary, 
computational, but nevertheless reasonable short and  elegant proof
which is simpler than Fialowski's 
original calculations \cite{Fianote}.

As the Witt and Virasoro algebra are of fundamental importance inside
of mathematics and in the applications, and up to now 
there is no complete published proof,
the presented proof is for sure worth-while to publish.
The proof avoids the heavy machinery  of 
Tsujishita, Reshetnikov,  and Goncharova.%
\footnote{As Fuks writes in his book \cite[p.119]{Fuks}
``Theorem 2.3.1 [Goncharova's theorem] still remains one of the most
difficult theorems in the cohomology theory of infinite-dimensional
Lie algebras.''}

Furthermore the presented
proof has the advantage that by its algebraic nature it will
be valid for every field $\K$ of characteristic zero.
Moreover, this is one of the rare occasions where algebraic
cohomology is calculated.
From the very beginning it was desirable to have a purely algebraic 
proof of the vanishing of the algebraic cohomology.

\medskip
In \refS{algebra} we give the definition of both the
Witt and Viraosoro algebra and make some remarks on their
graded structure. 
The graded structure 
will play an important role in the article. 

In \refS{cohomo} 
after recalling the definition of general Lie algebra cohomology 
the 
2-cohomology of a Lie algebra  with values in the
adjoint module is considered in more detail. 
Some facts about its relation to deformations 
of the algebra are
quoted to allow to judge the importance of this cohomology
\cite{FiaSchlaffo}.
The main result about the vanishing of the second 
cohomology and the rigidity for the Witt and Virasoro algebra is 
formulated in \refT{main} and \refC{rigid}.
The cohomology considered in this article is algebraic cohomology
without  any restriction on the cocycles.

In \refS{witt} we use the graded structure of the
algebras to decompose their cohomology into
graded pieces. It is quite easy then to show
that for degree $d\ne 0$, the degree $d$  parts will vanish.
Hence, everything is reduced to degree zero.
This is due to the fact that the grading is induced by the
action of a special element. Here it is the element $e_0$.

The vanishing of the 
degree zero part 
 is  more involved. It will be presented for the 
Witt algebra 
in \refS{zero}. 
It is the computational core of the article, but 
the computations are elementary.

In \refS{vira} we extend this to the 
Virasoro algebra $\V$. We show that the vanishing of the cohomology
for $\W$ implies the same for its central extension $\V$.

\medskip
As the core of the proof is elementary, I want to keep this spirit
throughout the article. Hence, it will be rather self-contained and
elementary. 
We will not use things like long exact cohomology sequences, etc.
The proofs are purely algebraic and the vanishing and  
rigidity theorems are true  over every field 
$\K$ of characteristic zero.

\section{The algebras}\label{S:algebra}

The {\it Witt algebra} $\W$ is the Lie algebra  
generated as vector space over a field $\K$ by the
elements 
$\{e_n\mid n\in \mathbb{Z}\}$ with 
Lie structure  
\begin{equation}\label{E:Wstruct}
[e_n,e_m]=(m-n)e_{n+m},\quad n,m\in\Z.
\end{equation}
As in the cohomology computations we have to divide by
arbitrary integers we assume for the whole article that
$\chara(\K)=0$.
\begin{remark}
In conformal field theory $\K$ will be $\C$ 
and the Witt algebra
can be realized as complexification of the Lie algebra of
polynomial vector fields $Vect_{pol}(S^1)$ on the circle $S^1$,
which is a subalgebra of $Vect(S^1)$, 
die Lie algebra of all $C^\infty$ vector
fields  on the circle. 
In this realization $e_n=\exp(\mathrm{i} \, n\, \varphi)
\frac {d}{d\varphi}$.
The Lie product is the usual bracket of vector fields. 

An alternative realization is given as the
algebra of meromorphic vector fields on the Riemann sphere
$\Pro^1(\C)$ which are holomorphic outside $\{0\}$ and 
$\{\infty\}$. In this realization 
$e_n=z^{n+1}\frac {d}{dz}$.
\end{remark}

A very important fact is that the Witt algebra is a $\Z$-graded Lie
algebra.
We define the degree by 
setting $\deg(e_n):=n$, then the Lie product 
between  elements
of degree $n$ and of degree $m$ is of degree $n+m$ 
(if nonzero).
The homogeneous spaces $\W_n$ of degree $n$ are one-dimensional with basis
$e_n$.
Crucial 
for the following is the additional fact
that the eigenspace decomposition of the
element $e_0$, acting via the adjoint action on $\W$ 
coincides with the decomposition into homogeneous 
subspaces. This follows from
\begin{equation}\label{E:eigen}
[e_0,e_n]=n\,e_n=\deg(e_n)\,e_n.
\end{equation}
Another property, which will play a role,  is
that $[\W,\W]=\W$. That means that $\W$ is a perfect Lie algebra.
In fact, 
\begin{equation}\label{E:perfect}
e_n=\frac 1n[e_o,e_n], \quad n\ne 0,\qquad
e_0=\frac 12[e_{-1},e_{1}].
\end{equation}

The {\it Virasoro algebra} $\V$ is the universal  one-dimensional
central extension of $\W$. 
As vector space it is the direct sum $\V=\K\oplus \W$.
If we set for $x\in\W$, 
$\hat x:=(0,x)$, and $t:=(1,0)$
then 
its basis elements are $\hat e_n, \ n\in\Z$ and 
$t$ with the  Lie product
\begin{equation}\label{E:Vstruct}
[\hat e_n,\hat e_m]=(m-n)\hat e_{n+m}-\frac 1{12}(n^3-n)\delta_n^{-m}\,t,\quad 
\quad
[\hat e_n,t]=[t,t]=0,
\end{equation}
for all $n,m\in\Z$.
Here $\delta_k^l$ is the Kronecker delta which is equal to 1 if
$k=l$, otherwise zero.
If we set $\deg(\hat e_n):=\deg(e_n)=n$ and $\deg(t):=0$ then 
$\V$ becomes a graded algebra. 
Let $\nu$ be the Lie homomorphism mapping the 
central element $t$ to $0$ and $\hat x$ to $x$
inducing the following short exact sequence
of
Lie algebras 
\begin{equation}\label{E:cext}
\begin{CD}
0 @>>> \K  @>>>\V @>\nu>> \W @>>> 0\ .
\end{CD}
\end{equation} 
This sequence does not split, i.e. 
it is a non-trivial central extension.

In some abuse of notation we identify the element $\hat x\in\V$ with
$x\in\W$ and after identification 
we have $\V_n=\W_n$ for $n\ne 0$ and $\V_0={\langle e_0,t\rangle}_{\K}$.
Note that the relation \refE{eigen} 
inducing the eigenspace decomposition for the grading element
$\hat e_0=e_0$ remains true.

The expression
$\frac 1{12}(n^3-n)\delta_n^{-m}$ is the defining
cocycle for the central extension. This form is given in a 
standard normalisation -- others are possible.
It is a Lie algebra two-cocycle of $\W$ with values in the trivial 
module.
The equivalence classes of central extensions are in 1:1
correspondence to the cohomology classes $\Ho^2(\W,\K)$, see the
next section for more details.
It is well-known that
$\dim \Ho^2(\W,\K)=1$, and that the class of the above
cocycle is a generator.

\section{Cohomology and Deformations}\label{S:cohomo}

Let us recall for completeness and further reference the definition of the
{\it Lie algebra cohomology} of a Lie algebra $W$ with
values in a Lie module $M$ over $W$.
We denote the Lie module structure by $W\times M\to M, \
(x,m)\mapsto x\ldot m$.
A $k$-cochain is an alternating $k$-multilinear map
$W\times W\times\cdots \times W\to M$ ($k$ copies of $W$).
The vector space of $k$-cochains is denoted by 
$C^k(W;M)$.
Here we will deal exclusively with algebraic cohomology.

Next we have the family of coboundary operators 
\begin{equation}
\delta_k: C^k(W;M)\to  C^{k+1}(W;M),\quad k\in\N,\quad
\text{with}\quad 
\delta_{k+1}\circ\delta_k=0.
\end{equation}
Here we will  only consider the second cohomology
\begin{equation}\label{E:2cocgen}
\begin{aligned}
\delta_2\psi(x,y,z)&:=
\psi([x,y],z)+
\psi([y,z],x)+
\psi([z,x],y)
\\&\qquad
-x\ldot \psi(y,z)
+y\ldot\psi(x,z)
-z\ldot\psi(x,y),
\\
(\delta_1\phi)(x,y)&:=
\phi([x,y])-x\ldot\phi(y) +y\ldot\phi(x).
\end{aligned}
\end{equation}
A $k$-cochain $\psi$ is called   a {\it k-cocycle} if it lies in the kernel of
the $k$-coboundary operator $\delta_k$.
It is called a   {\it k-coboundary} if
it lies in the image of the $(k-1)$-coboundary operator.

By $\delta_{k}\circ\delta_{k-1}=0$ the 
vector space quotient of cocycles modulo
coboundaries
is well-defined. It is called the vector space of 
{\it $k$-Lie algebra cohomology} of $W$ with values in the module $M$. It is
denoted by  $\Ho^k(W;M)$.
Two cocycles which are in the same cohomology class are called
{\it cohomologous}.

\medskip
The trivial module is $\K$ with the Lie action
$x\ldot m=0$, for all $x\in W$ and $m\in \K$.
The second cohomology with values in the trivial module
classifies
equivalence classes of 
central extensions of $W$. 
It is well-known that for the Witt algebra $\W$ we have
$\dim\Ho^2(\W;\K)=1$ and that the class of 
the cocycle defining $\V$  gives a basis.

\medskip
The second cohomology of $W$ with values in the adjoint module,
$\Ho^2(W;W)$, i.e. with module structure
$x\ldot y:=[x,y]$, is the cohomology to be studied here. 
As we will need the formula for explicit calculations later let me
specialize the 2-cocycle condition \refE{2cocgen}
\begin{equation}\label{E:2cob}
\begin{aligned}
\delta_2\psi(x,y,z)&:=
0=\psi([x,y],z)+
\psi([y,z],x)+
\psi([z,x],y)
\\&\qquad
-[x,\psi(y,z)]
+[y,\psi(x,z)]
-[z,\psi(x,y)].
\end{aligned}
\end{equation}
A 2-cocycle will be a coboundary if  
it lies in the image of the (1-)coboundary operator, i.e. 
there exists a linear map $\phi:W\to W$ such that
\begin{equation}\label{E:1cob}
\psi(x,y)=(\delta_1\phi)(x,y):=
\phi([x,y])-[\phi(x),y]-[x,\phi(y)].
\end{equation}
 
\bigskip

The second cohomology $\Ho^2(W,W)$ is  related to the
deformations of the Lie algebra $W$.
\begin{enumerate}
\item
$\Ho^2(W,W)$ classifies infinitesimal deformations of $W$ 
up to equivalence (Gerstenhaber \cite{Ger}).
\label{it1}
\item
If $\dim \Ho^2(W,W)<\infty$ then 
there exists a versal formal  family for the 
formal deformations of $W$  whose base  is formally embedded 
into $\Ho^2(W,W)$. This is due to Fialowski \cite{Fiaproc}, and
Fialowski and Fuks \cite{FiaFuMini}.
\end{enumerate} 
A Lie algebra $W$ is called \emph{rigid} if every deformation
is locally equivalent to the trivial family.
Hence, 
if $\Ho^2(W,W)=0$ then $W$ 
is infinitesimally and formally rigid.
See \cite{FiaSchlaffo} for more information on the connection 
between cohomology and deformations.

\bigskip
After having recalled the general definitions, I formulate the
main result of this article.
\begin{theorem}\label{T:main}
Both the second cohomology of the Witt algebra $W$ 
and of the Virasoro algebra $\V$ (over a field $\K$ with $\chara(\K)=0$) 
with values
in the adjoint module vanishes, i.e.
\begin{equation}\label{E:main}
\Ho^2(\W;\W)=\{0\}, \qquad \Ho^2(\V;\V)=\{0\}.
\end{equation}
\end{theorem}
\begin{corollary}\label{C:rigid}
Both $\W$ and $\V$ are 
formally and infinitesimally 
rigid.
\end{corollary}
\noindent
I refer to \refS{intro}, the Introduction, for
the history of this theorem. Here I only like to repeat that
the first proof (at least in the Witt case) 
was given by Alice Fialowski by very cumbersome 
unpublished calculations 
\cite{Fianote}
\footnote{Note added after publication of the current article:
In the meantime they are published, see \cite{Fianote} for details}. 
The proof which I present in the following is elementary, but much
more accessible and elegant as the original proof.


\section{The degree decomposition of the cohomology}\label{S:witt}

Let $W$ be an arbitrary $\Z$-graded Lie algebra, i.e.
\begin{equation}
W=\bigoplus_{n\in\Z}W_n.
\end{equation}
Recall that we consider algebraic cohomology, i.e. 
our 2-cochains $\psi\in C^2(W;W)$ are 
arbitrary alternating bilinear maps in
the usual sense, i.e. for all $v,w\in W$ 
the cochain $\psi(v,w)$ will be a finite
linear combination (depending on $v$ and $w$) 
of basis elements in $W$.
We call a $k$-cochain $\psi$ homogeneous of degree $d$ if
there exists a $d\in\Z$ such that for all $i_1,i_2,\ldots,i_k\in\Z$
and homogeneous elements $x_{i_l}\in W$, of $\deg(x_{i_l})={i_l}$, for
$l=1,\ldots, k$ we have that
\begin{equation}
\psi(x_{i_1},x_{i_2},\ldots, x_{i_k})\ \in \ W_n,\quad 
\text{with}\quad n=\sum_{l=1}^k i_l+d.
\end{equation}
The corresponding subspace of degree $d$\; homogeneous k-cochains is
denoted by $C^k_{(d)}(W;W)$.
Every $k$-cochain can be written as a formal
infinite sum 
\begin{equation}
\psi=\sum_{d\in\Z}\psi_{(d)}.
\end{equation}
Note that evaluated for a fixed $k$-tuple  of elements only a finite number of
the summands  will produce values different from zero.

An inspection of 
the coboundary operators 
\refE{2cob} and \refE{1cob} for homogeneous elements $x$, $y$, $z$ 
shows
\begin{proposition}\label{P:41}
The coboundary operators $\delta_k$ are operators of degree zero, i.e.
applied to a $k$-cocycle of degree $d$ they will produce
a $(k+1)$-cocycle also of degree $d$.
\end{proposition}

In the following we will concentrate on $k=2$ or $k=1$.
If 
$\psi=\sum_{d} \psi_{(d)}$ is  a 2-cocycle then
$\delta_2\psi=\sum_{d} \delta_2\psi_{(d)}=0$.
By \refP{41}  $\delta_2\psi_{(d)}$ is either zero or of degree $d$.
As we sum over different degrees 
and the terms cannot cancel if different from zero we obtain that
$\psi$ is cocycle if and only if 
all degree $d$ components $\psi_{(d)}$ will be individually
2-cocycles.
Moreover, if $\psi_{(d)}$ is  2-coboundary, i.e.   
$\psi_{(d)}=\delta_1\phi$ with a 1-cochain $\phi$, then we can find 
another 1-cochain $\phi'$ of degree  $d$
such that $\psi_{(d)}=\delta_1\phi'$.

We summarize as follows.
Every cohomology class $\alpha\in\Ho^2(W;W)$ can be decomposed as
formal sum 
\begin{equation}
\alpha=\sum_{d\in\Z}\alpha_{(d)},
\qquad
\alpha_{(d)}\in \Ho_{(d)}^2(W;W),
\end{equation}
where the latter space consists of 
classes of cocycles of degree $d$ modulo coboundaries of degree
$d$.

\medskip
For the rest of this section 
let $W$  be either $\W$ or $\V$ and $d\ne 0$.
We will show that for these algebras 
the cohomology spaces of degree $d\ne 0$ will
vanish.
The degree zero case needs a more involved treatment and will
be done in the next sections.
We start with a cocycle of degree $d\ne 0$ and make first a
cohomological change $\psi'=\psi-\delta_1\phi$ with
\begin{equation} 
\phi: W\to W,\quad x\mapsto \phi(x)=\frac {\psi(x,e_0)}{d}.
\end{equation}
Recall $e_0$ is the element of either $\W$ or $\V$ which 
gives the degree decomposition.
This implies (note that by definition $\phi(e_0)=0$)
\begin{equation}\label{E:cocdad}
\begin{aligned}
\psi'(x,e_0)&=\psi(x,e_0)-(\delta_1\phi)(x,e_0)=
\psi(x,e_0)-
\phi([x,e_0])+[\phi(x),e_0]
\\
&=
d\,\phi(x)+\deg(x)\phi(x)-(\deg(x)+d)\phi(x)=
0.
\end{aligned}
\end{equation}
We evaluate \refE{2cob} for the cocycle $\psi'$ on the
triple $(x,y,e_0)$ and leave out  the cocycle
values which vanish due to 
\refE{cocdad}:
\begin{equation}\label{E:cocd0}
\begin{aligned}
0&=
\psi'([y,e_0],x)
+\psi'([e_0,x],y)
-[e_0,\psi'(x,y)]
\\
&=
(\deg(y)+\deg(x)-(\deg(x)+\deg(y)+d))
\psi'(x,y)=-d\psi'(x,y).
\end{aligned}
\end{equation}
As $d\ne 0$ we obtain $\psi'(x,y)=0$ for all $x,y\in W$.
We conclude
\begin{proposition}\label{P:reduce}
The following hold:
\smallskip

(a) \qquad $\Ho_{(d)}^2(\W;\W)=\Ho_{(d)}^2(\V;\V)=\{0\}$,\quad
 for $d\ne 0$. 

\smallskip
(b)  \qquad $\Ho^2(\W;\W)=\Ho_{(0)}^2(\W;\W)$,\quad
\qquad $\Ho^2(\V;\V)=\Ho_{(0)}^2(\V;\V)$.
\end{proposition}
\begin{remark}
In fact the arguments also  work for every $\Z$-graded Lie
algebra
$W$ 
for which 
there exists an element $e_0$ such  that the homogeneous spaces
$W_n$  
are just the eigenspaces of $e_0$ under the adjoint action 
to the eigenvalue $n$. 
Such Lie algebras are called internally graded.
See also  Theorem 1.5.2  in \cite{Fuks}.
\end{remark}

\section{Degree zero for the Witt algebra}\label{S:zero}
It remains the degree zero part. In this section we 
consider only the Witt algebra.
Recall that  the homogeneous subspaces of 
degree $n$ are one-dimensional and generated by $e_n$.
Hence, a  degree zero cocycle can be written as
$\psi(e_i,e_j)=\psi_{i,j} e_{i+j}$ and if it 
is a coboundary  then it can
be given as a coboundary of a linear form of degree zero:
$\phi(e_i)=\phi_ie_i$. The systems of $\psi_{i,j}$ and $\phi_i$ 
for $i,j\in\Z$ fix $\psi$ and $\phi$ completely.
If we evaluate   \refE{2cob} for the triple $(e_i,e_j,e_k)$  
we get for the coefficients 
\begin{equation}\label{E:cocycgen}
\begin{aligned}
0=&(j-i)\psi_{i+j,k}-(k-i)\psi_{i+k,j}+(k-j)\psi_{j+k,i}
\\
& -(j+k-i)\psi_{j,k}+(i+k-j)\psi_{i,k}
-(i+j-k)\psi_{i,j}.
\end{aligned}
\end{equation}
For the coboundary we obtain
\begin{equation}\label{E:cobound}
(\delta \phi)_{i,j}=(j-i)(\phi_{i+j}-\phi_j-\phi_i).
\end{equation}
Hence,  $\psi$ is a coboundary if and only if there exists a system of
$\phi_k\in\K$, $k\in\Z$ such that
\begin{equation}\label{E:coboundc}
\psi_{i,j}=(j-i)(\phi_{i+j}-\phi_j-\phi_i),\quad \forall i,j\in\Z.
\end{equation}
\noindent
A degree zero 1-cochain $\phi$ will be  a 1-cocycle  
(i.e. $\delta_1 \phi=0$)
if and only if 
\begin{equation}
\phi_{i+j}-\phi_j-\phi_i=0.
\end{equation} 
This has the solution $\phi_i=i\,\phi_1, \forall i\in\Z$.
Hence, 
given a $\phi$ we can always find a $\phi'$ with   $(\phi')_1=0$
and  $\delta_1 \phi=\delta_1 \phi'$.
In the following we will always choose such a $\phi'$ for our
2-coboundaries.

{\bf  Step 1:} We make a cohomological change

\noindent
We start with a  2-cocycle $\psi$ given by the system of $\psi_{i,j}$ 
and will modify it by adding a coboundary $\delta_1\phi$ with 
suitable $\phi$ to obtain $\psi'=\psi-\delta_1\phi$.
We will determine $\phi$  inspired by the intended relation \refE{coboundc}
\begin{equation}\label{E:c1}
\psi_{i,1}=(1-i)(\phi_{i+1}-\phi_1-\phi_i)= (1-i)(\phi_{i+1}-\phi_i).
\end{equation} 
Note that we could put $\phi_1=0$ by our normalization.

\noindent
(a) Starting from $\phi_0:=-\psi_{0,1}$ we set in descending order
for $i\le -1$
\begin{equation}
\phi_i:=\phi_{i+1}-\frac 1{1-i}\,\psi_{i,1}.
\end{equation}

\noindent
(b) $\phi_2$ cannot be fixed by \refE{c1}, instead we use
\refE{coboundc}
\begin{equation}
\psi_{-1,2}=3(-\phi_2-\phi_{-1}), \quad\text{yielding}\quad
\phi_2:=-\phi_{-1}-\frac 13\psi_{-1,2}.
\end{equation}
Then we have ${\psi'}_{-1,2}=0$.

\noindent
(c) We use again \refE{c1} to calculate recursively
in ascending order $\phi_i,\  i\ge 3$
by
\begin{equation}
\phi_{i+1}:=\phi_{i}+\frac 1{1-i}\,\psi_{i,1}.
\end{equation}
For the cohomologous cocycle $\psi'$ we obtain by  construction
\begin{equation}
\psi'_{i,1}=0,\quad\forall i\in\Z, \quad\text{and}
\quad \psi'_{-1,2}=\psi'_{2,-1}=0.
\end{equation}

\bigskip
{\bf Step2:} Show that $\psi'$ is identical zero.

\noindent
To avoid cumbersome notation we will denote the cohomologous
cocycle by $\psi$.
\begin{lemma}\label{L:vanish}
Let $\psi$ be a 2-cocycle of degree zero such that
$\psi_{i,1}=0,\forall i\in\Z$ and $\psi_{-1,2}=0$, then
$\psi$ will be identical zero.
\end{lemma}
Before we proof the lemma we use it to show
\begin{proof} (Witt part of \refT{main}.)
By \refP{reduce} it is enough to consider  degree zero cohomology.
By the cohomological change done in Step 1
every degree zero cocycle $\psi$ is cohomologous to
a cocycle fulfilling the conditions of \refL{vanish}. 
But such a cocycle vanishes by the lemma. Hence, the original cocycle 
$\psi$ is cohomological trivial. 
\end{proof}

\begin{proof} (\refL{vanish}.)
First, we note two special cases of \refE{cocycgen}
which will be useful in the following.
For the index triple $(i,-1,k)$ we obtain
\begin{equation}\label{E:cockm1}
\begin{aligned}
0=&-(i+1)\psi_{i-1,k}-(k-i)\psi_{i+k,-1}+(k+1)\psi_{k-1,i}
\\
& -(-1+k-i)\psi_{-1,k}+(i+k+1)\psi_{i,k}
-(i-1-k)\psi_{i,-1}.
\end{aligned}
\end{equation}
For the triple $(i,1,k)$, and ignoring terms of the type $\psi_{i,1}$ which
 are zero by assumption, we obtain
\begin{equation}\label{E:cock1}
0=(1-i)\psi_{i+1,k}+(k-1)\psi_{k+1,i}
+(i+k-1)\psi_{i,k}.
\end{equation}
We will consider $\psi_{i,m}$ for certain values of $|m|\le 2$ and finally 
make ascending and descending induction on $m$. 
We will call the coefficient $\psi_{i,m}$ coefficients of level $m$
(and of level $i$ by  antisymmetry).
By assumption the cocycle values of level 1 are all zero. 

\smallskip
\noindent
$\boxed{m=0}$
\smallskip

\noindent
By the antisymmetry we have $\psi_{1,i}=-\psi_{i,1}=0$.
In \refE{cock1} we consider $k=0$
this gives
\begin{equation}\label{E:c0}
0=(1-i)\left(\psi_{i+1,0}-\psi_{i,0}\right).
\end{equation}
Starting from $\psi_{1,0}=0$ this implies for 
$i\le 0$ that $\psi_{i,0}=0$
and for $i\ge 3$ that $\psi_{i,0}=\psi_{2,0}$.
Next we consider \refE{cockm1} for $k=2, i=0$ and obtain
\begin{equation}
0=-\psi_{-1,2}-2\psi_{2,-1}+3\psi_{1,0}
-\psi_{-1,2}+3\psi_{0,2}+3\psi_{0,-1}.
\end{equation}
The $\psi_{-1,2}$ terms cancel and we know already $\psi_{1,0}=\psi_{-1,0}=0$,
hence $\psi_{2,0}=0$. 
This implies
\begin{equation}
\psi_{i,0}=0\quad \forall i\in\Z.
\end{equation}

\smallskip
\noindent
$\boxed{m=-1}$
\smallskip

\noindent
In \refE{cock1} we set $k=-1$ and obtain 
(with $\psi_{0,i}=0$)
\begin{equation}
-(i-1)\,\psi_{i+1,-1}+(i-2)\,\psi_{i,-1}=0.
\end{equation}
Hence,
\begin{equation}
\psi_{i,-1}=\dfrac {i-1}{i-2}\;\psi_{i+1,-1},\quad \text{for}\ i\ne 2,
\qquad
\psi_{i+1,-1}=\dfrac {i-2}{i-1}\;\psi_{i,-1},\quad
\text{for}\ i\ne 1.
\end{equation}
The first formula implies 
starting from 
$\psi_{1,-1}=-\psi_{-1,1}=0$  that
$\psi_{i,-1}=0$, for all $i\le 1$.
The second formula for $i=2$ implies $\psi_{3,-1}=0$ and hence
$\psi_{i,-1}=0$  for $i\ge 3$.
But by assumption $\psi_{2,-1}=\psi_{-1,2}=0$.
Hence,
\begin{equation}
\psi_{i,-1}=0\quad \forall i\in\Z.
\end{equation}

\smallskip
\noindent
$\boxed{m=-2}$
\smallskip

\noindent
We plug the value $k=-2$ into \refE{cock1} and get for the
terms not yet identified as zero
\begin{equation}
(1-i)\psi_{i+1,-2}+(i-3)\psi_{i,-2}=0.
\end{equation}
This yields 
\begin{equation}
\psi_{i+1,-2}=\dfrac {i-3}{i-1}\;\psi_{i,-2},\quad \text{for}\ i\ne 1,
\qquad
\psi_{i,-2}=\dfrac {i-1}{i-3}\;\psi_{i+1,-2},\quad \text{for}\ i\ne 3.
\end{equation}
From the first formula we get
$\psi_{3,-2}=-\psi_{2,-2}$, 
$\psi_{4,-2}=0\cdot\psi_{3,-2}$, and hence
$\psi_{i,-2}=0$, for all $i\ge 4$.
\newline
From the second formula we get starting from $\psi_{1,-2}=0$ that
$\psi_{i,-2}=0$ for $i\le 1$.
\newline
Altogether, $\psi_{i,-2}=0$ for $i\ne 2,3$. The value of
 $\psi_{2,-2}=-\psi_{3,-2}$ stays undetermined for the moment.

\smallskip
\noindent
$\boxed{m=2}$
\smallskip

\noindent
We start from \refE{cockm1} for $k=2$ and recall that terms of
levels 0, 1, -1 are zero. This gives
\begin{equation}
-(i+1)\,\psi_{i-1,2}+(i+3)\,\psi_{i,2}=0.
\end{equation}
Hence,
\begin{equation}
\psi_{i,2}=\dfrac {i+1}{i+3}\;\psi_{i-1,2},
\quad \text{for}\ i\ne -3,
\qquad
\psi_{i-1,2}=\dfrac {i+3}{i+1}\;\psi_{i,2},\quad \text{for}\ i\ne -1.
\end{equation}
From the first formula we start from $\psi_{-1,2}=0$ and get 
$\psi_{i,2}=0$, $\forall i\ge -1$. From the second we get
$\psi_{-3,2}=-\psi_{-2,2}$, then $\psi_{-4,2}=0$ and then altogether
$\psi_{i,2}=0$ for all $i\ne -2,-3$. 
The value
$\psi_{-3,2}=-\psi_{-2,2}$ stays undetermined for the moment.

To find it 
we consider the index triple $(2,-2,4)$ in 
\refE{cocycgen} and obtain after leaving out terms which are 
obviously zero
\begin{equation}
0=-2\,\psi_{6,-2}-8\,\psi_{4,2}+4\,\psi_{2,-2}.
\end{equation}
From the level $m=2$ discussion we get $\psi_{4,2}=0$, from
$m=-2$ we get $\psi_{6,-2}=0$. This shows that 
$\psi_{2,-2}=\psi_{3,-2}=\psi_{-3,2}=0$ and we can conclude
\begin{equation}
\psi_{i,-2}=\psi_{i,2}=0,\qquad\forall i\in\Z.
\end{equation}

\smallskip
\noindent
$\boxed{m<-2}$
\smallskip

\noindent
We make induction assuming it is true for $m=2,1,0,-1,-2$.
We start from \refE{cockm1} for $k$ and put the $k-1$ level 
element on the l.h.s.
By this
\begin{equation}
(k+1)\psi_{i,k-1}=\text{terms of level $k$ and $-1$}.
\end{equation}
By induction the terms on the r.h.s. are zero. 
Note that in this region $k<-1$, hence $k+1\ne 0$,
and $\psi_{i,k-1}=0$, too.

\smallskip
\noindent
$\boxed{m>2}$
\smallskip

\noindent
Again we make induction.
Starting from 
\refE{cock1} we get
\begin{equation}
(k-1)\,\psi_{i,k+1}=(1-i)\,\psi_{i+1,k}+(i+k-1)\,\psi_{i,k}.
\end{equation}
As $k\ge 2$ the value of $1-k\ne 0$ and 
we get by induction trivially the statement for $k+1$.

Altogether we obtain \ $\psi_{i,k}=0,\forall i,k\in\Z$.
\end{proof}
The presented calculation is completely different and much simpler
as in \cite{Fianote} 

\section{Extension to the Virasoro algebra}\label{S:vira}

Here we show in an elementary way that also for the Virasoro algebra
\begin{equation}
\Ho^2(\V,\V)=0. 
\end{equation}
Hence $\V$ is also infinitesimally and
formally rigid.
This shows the Virasoro part of \refT{main}.
It is our intention not to use any higher techniques,
but see  \refR{rem2} at the end of this section.

First recall that by \refP{reduce} it is enough to 
consider degree zero cocycles. 
We start with a degree zero 2-cocycle $\psi:\V\times \V\to \V$  of the
Virasoro algebra.
If we apply the Lie homomorphism $\nu$ we get 
the bilinear map $\nu\circ\psi$ which we restrict  to
$\W\times\W$
\begin{equation}
\psi'=\nu\circ\psi_{|}:\W\times\W\to \W.
\end{equation}
Unfortunately,  in general $\psi'$ will not be a 2-cocycle 
for the Witt algebra. We have to be careful as
the Lie product for $\V$   differs from that 
of $\W$ by multiples of the central element.
But we are allowed to make cohomologous changes.
\begin{proposition}\label{P:restrcohomo}
Given a cocycle $\psi\in C^2(\V,\V)$ there exists a cohomologous one 
$\tilde \psi\in C^2(\V,\V)$
 such that the bilinear map 
$\psi'=\nu\circ\tilde\psi\in 
 C^2(\W,\W)$.
\end{proposition}
\begin{proof}
Let $x,y,z\in\W$. We have $[x,y]_{\V}=[x,y]_{\W}+\alpha(x,y)\cdot t$.
We consider \refE{2cob} with the bracket $[.,.]_{\V}$ and rewrite
it  in terms of  $[.,.]_{\W}$ (but drop the index) and $\alpha$.
For the second group of terms we use that $\nu$ is
a Lie homomorphism and they will not see the central elements. Only 
in the first group they will play a role.
We get
\begin{equation}\label{E:2cobvv}
\begin{aligned}
0&=
\psi'([x,y],z)+
\psi'([y,z],x)+
\psi'([z,x],y)
\\
&\qquad
+\alpha({[x,y],z})\,\nu\circ\psi(t,z)
+\alpha({[y,z],x})\,\nu\circ\psi(t,x)
+\alpha({[z,x],y})\,\nu\circ\psi(t,y)
\\&\qquad
-[x,\psi'(y,z)]
+[y,\psi'(x,z)]
-[z,\psi'(x,y)].
\end{aligned}
\end{equation}
We will make a cohomological change for the cocycle $\psi$ in $\V$ 
such that $\psi(t,z)$ will have only central terms, i.e. 
$\nu\circ\psi(t,z)$ will vanish.
Restricting it to
$\W\times \W$ and projecting it by $\nu$ 
will define
a 2-cocycle for $\W$.

As $\psi$ is a degree zero cocycle, we have $\psi(e_1,t)=a\, e_1$.
We set $\phi(t):=a \, e_0$, and $\phi(e_n):=0$ for all $n\in\Z$.
Let $\tilde\psi:=\psi-\delta_1\phi$.
We calculate
\begin{equation}
\begin{aligned}
\tilde\psi(e_1,t)=\psi(e_1,t)-(\delta_1\phi)(e_1,t)
&=
a\, e_1 - \phi([e_1,t])+[e_1,\phi(t)]+[\phi(e_1),t]
\\
&=a\,e_1-0-a\,e_1+0=0.
\end{aligned}
\end{equation}
The coefficients $a_n$ (and $b$, but it will not
play any role) are given by  
\begin{equation}
\tilde\psi(e_n,t)=a_ne_n,\ n\ne 0,\qquad 
\tilde\psi(e_0,t)=a_0e_n+b\,\cdot t.
\end{equation}
Note that $a_1=0$. If we evaluate \refE{2cob} for the
triple $(e_n,e_m,t)$ and leave out terms which are  
zero due to the fact, that $t$ is central we get
\begin{equation}\label{E:shift}
0=\tilde\psi([e_n,e_m]_{\V},t)
-[e_n,\tilde\psi(e_m,t)]_{\V}
+[e_m,\tilde\psi(t,e_n)]_{\V}.
\end{equation}
This implies (by evaluating the coefficients at the 
element $e_{m+n}$)
\begin{equation}
(m-n)(a_{n+m}-a_m-a_n)=0.
\end{equation}
If we plug in $m=1$ and use $a_1=0$ we get $a_n=0$ for all $n\le 1$
and
$a_n=a_2$ for all $n\ge2$. Now plugging in $m=2$, $n=-2$ 
we obtain $a_2=-a_{-2}=0$. Hence altogether 
$a_n=0$ for all $n$, and $\nu\circ\tilde\psi(w,t)=0$, $\forall
w\in\V$.
\end{proof} 
Hence, after a cohomologous change we may assume that our 
restricted and 
projected $\psi'$ will be a 2-cocycle for the
algebra $\W$ with values in the adjoint module $\W$. In the last
section we showed that it is cohomologically trivial, i.e. there 
exists a $\phi':\W\to\W$ such that
$\delta_1^{\W}\phi'=\psi'=\nu\circ\psi$.
We denote for a moment the corresponding coboundary operators of the
two Lie algebras by $\delta^{\W}$ and $\delta^{\V}$.
By setting $\phi(t):=0$ we extend $\phi'$ 
to a linear map $\phi:\V\to\V$.
In particular, $\nu\circ\phi=\phi'$ if restricted to
$\W$.

The 2-cocycle $\hat\psi=\psi-\delta_1^{\V}\phi$ will be a
cohomologous  cocycle for $\V$. If we apply 
$\nu$  then
\begin{equation}
\nu\hat\psi=\nu\psi-\nu\delta_1^{\V}\phi=
\psi'-\delta_1^{\W}\phi'=0.
\end{equation}
Hence, $\hat\psi$ takes values in the kernel of $\nu$, i.e.
\begin{equation}
\hat\psi:\V\times \V\to \K\cdot t,
\end{equation}
with $t$ the central element.

\medskip
This implies that 
it is enough to
show that every  class
$\psi:\V\to\V$ with values in the central ideal $\K\cdot t$ will be a
 coboundary 
to show the
vanishing of $\Ho^2(\V,\V)$.
In the following let $\psi$ be of this kind.

\medskip
As $\psi(x,y)=\psi_{x,y}t$ will be central the 2-cocycle condition \refE{2cob} 
will reduce to 
\begin{equation}\label{E:cocred}
(\delta_2\psi)(x,y,z)=
\psi([x,y],z)+
\psi([y,z],x)+
\psi([z,x],y)=0.
\end{equation}
The coboundary condition for $\phi:\V\to \K\cdot t$ reduces to
\begin{equation}\label{E:cocred1}
(\delta_1\phi)(x,y)=\phi([x,y]).
\end{equation}
We can reformulate this as that the component function 
$\psi_{x,y}:\V\times\V\to\K$ is a Lie algebra 2-cocycle for $\V$ with
values in the trivial module.
\begin{lemma}\label{L:van}
We have
\begin{equation}
\psi(x,t)=0,\quad
\forall x\in\V.
\end{equation}
\end{lemma} 
\begin{proof}
We evaluate \refE{cocred} for $(e_i,e_j,t)$.
As $t$ is central $[e_j,t]=[t,e_i]=0$, hence only the first term 
in \refE{cocred} 
will
survive and we get
$(j-i)\psi(e_{i+j},t)=0$. Choosing $j=0$ we obtain
$\psi(e_i,t)=0$ for $i\ne 0$, choosing $i=-1, j=+1$ we obtain also 
$\psi(e_0,t)=0$. As $\psi(t,t)=0$ is automatic and $\{e_n,n\in\Z,\ t\}$
is
a basis of $\V$ this shows the result.
\end{proof}
Next let $\phi:\V\to \K\cdot t$ be a linear map.
We write for the component functions
$\phi(e_i)=\phi_it$ and $\phi(t)=c\cdot t$.
If we evaluate the 
1-coboundary operator 
\refE{cocred1} for pairs of basis elements we get
\begin{equation}\label{E:cobred}
\delta_1\phi(e_i,e_j)=\phi([e_i,e_j])
=\left((j-i)\phi_{i+j}-\frac {1}{12}(i^3-i)\delta_{i}^{-j}c\right)t,\quad
i,j\in\Z,
\end{equation}
and by \refL{van} we get $\delta_1\phi(e_i,t)=0$, for all $i\in\Z$.

With respect to \refE{cobred} we choose $\phi$ such that 
the cohomologous cocycle $\psi'=\psi-\delta_1\phi$ fulfills 
\begin{equation}\label{E:trans}
\psi'(e_i,e_0)=0,\ \forall i\in\Z,\quad\text{and}\quad
\psi'(e_1,e_{-1})=
\psi'(e_2,e_{-2})=0.
\end{equation}
This is obtained by putting
\begin{equation}
\phi_i:=-\frac {1}{i}\psi_{e_i,e_0},\ i\in\Z, i\ne 0,
\quad
\phi_0:=-(1/2)\psi_{e_1,e_{-1}},
\quad
c:=(-2\,\psi_{e_2,e_{-2}}-8\phi_0).
\end{equation}
\begin{lemma}\label{L:zero}
The cocycle $\psi'$ is identically zero.
\end{lemma}
\begin{proof}
First we consider the 2-cocycle condition 
\refE{cocred} for the
triple $(e_i,e_j,e_0)$ and obtain using $\psi'(e_i,t)=0$ that
\begin{equation}
0=0+(-j)\psi'(e_j,e_i)+i\,\psi'(e_i,e_j)
=(i+j)\psi'(e_i,e_j).
\end{equation}
This yields that 
\begin{equation}
\psi'(e_i,e_j)=0 \quad \text{for $j\ne -i$}.
\end{equation}
Next we show by  induction that 
$\psi'(e_i,e_{-i})=0$ for all $i\ge0$ (and by antisymmetry also for 
$i\le 0$).
Note that  this is true for $i=0,1$ and $2$, see \refE{trans}.
Consider \refE{cocred} for the
triple $(e_n,e_{-(n-1)},e_{-1})$, $n>2$. After evaluation of the
Lie products we get
\begin{equation}
0=(-2n+1)\psi'(e_1,e_{-1})
+(n-2)\psi'(e_{-n},e_{n})
+(-n+1)\psi'(e_{n-1},e_{-(n-1)}).
\end{equation}
The first and last term vanishes by induction. As $n>2$ this implies
\begin{equation}
\psi'(e_{n},e_{-n})=0.
\end{equation}
As we have always 
$\psi'(x,t)=0$ this shows the lemma.
\end{proof}
Finally we obtain that the $\psi$ we started with was a
coboundary.
This implies indeed $\Ho^2(\V;\V)=\{0\}$.

\begin{remark}\label{R:rem2}
What has been done in this section can also be 
interpreted 
in the framework of long exact cohomology sequences
in Lie algebra cohomology.
The exact sequence \refE{cext} of Lie algebras is also 
a exact sequence of Lie modules over $\V$. For such sequences we have
a long exact sequence in cohomology. If we consider only 
level two we get
\begin{equation}\label{E:longcext}
\begin{CD}
\cdots @>>> \Ho^2(\V;\K)  
@>>>\Ho^2(\V ;\V)
@>\nu_*>>\Ho^2(\V; \W) @>>> \cdots\ .
\end{CD}
\end{equation} 
\refP{restrcohomo}
also shows that naturally $\Ho^2(\V; \W)\cong \Ho^2(\W; \W)$ 
(which is not a consequence from the 
existence of the long exact sequence).
The calculations of \refL{van} and \refL{zero} showed 
$ \Ho^2(\V;\K)=\{0\}$.
From \refS{zero} we know   $\Ho^2(\W; \W)=\{0\}$.
Hence, also
$\Ho^2(\V ;\V)=\{0\}$.
\end{remark}


\end{document}